\documentclass[oneside, a4paper,10pt,reqno]{amsart}

\usepackage[T1]{fontenc}
\usepackage{amsmath}	
\usepackage{enumitem} 
\usepackage{mathrsfs}
\usepackage{mathtools}
\mathtoolsset{showonlyrefs}
\usepackage[unicode]{hyperref}
\usepackage{bera}
\usepackage{bbm}
\newcommand{\D}{\mathbb{D}}   
\newcommand{\eps}{\varepsilon}
\newcommand{\N}{\mathbb{N}}   
\newcommand{\Q}{\mathbb{Q}}
\newcommand{\E}{\mathbb{E}}
\newcommand{\Z}{\mathbb{Z}}
\newcommand{\R}{\mathbb{R}}

\newtheorem{theorem}{Theorem}[section]
\newtheorem{proposition}[theorem]{Proposition}
\newtheorem{remark}[theorem]{Remark}

\newtheorem{corollary}[theorem]{Corollary}

\newtheorem{definition}[theorem]{Definition}
\newtheorem{lemma}[theorem]{Lemma}

\newtheorem{assumption}[theorem]{Assumption}

\makeatletter\@addtoreset{equation}{section}\makeatother

\makeatletter
\renewcommand{\section}{%
  \@startsection
      {section}
      {1}
      {0em}
      {\baselineskip}
      {0.5\baselineskip}
      {\normalfont\large\bfseries}}
\renewcommand{\subsection}{%
  \@startsection
      {subsection}
      {2}
      {0em}
      {\baselineskip}
      {0.25\baselineskip}
      {\normalfont\bfseries}}

\makeatother
\begin{document}

\title[Localization for RWRE]{Localization at the boundary for conditioned random walks in random environment in dimensions two and higher}

\author{Rodrigo Bazaes}
\address[Rodrigo Bazaes]{Facultad de Matem\'aticas\\
Pontificia Universidad Cat\'olica de Chile\\
Vicu\~na Mackenna 4860, Macul\\
Santiago, Chile}

\email{rebazaes@mat.uc.cl}

\thanks{The author has been supported by ANID-PFCHA/Doctorado Nacional/
2018-21180873.}
\thanks{\it Key words and phrases: random walks ; random environment ; localization.} 
\thanks{ AMS 2020 {\it subject classifications}. 60K37;82C41 ; 82D30}

\begin{abstract}
	We introduce the notion of \emph{localization at the boundary} for conditioned random walks in  i.i.d. and uniformly elliptic 
	random environment on $\Z^d$, in dimensions two and higher. Informally, this means that the walk spends a non-trivial amount of 
	time at some point $x\in \mathbb{Z}^{d}$ with $|x|_{1}=n$ at time $n$, for $n$ large enough. In dimensions two and three, we 
	prove localization for (almost) all walks. In contrast, for $d\geq 4$ there is a phase-transition for environments of the form $
	\omega_{\eps}(x,e)=\alpha(e)(1+\eps\xi(x,e))$, where $\{\xi(x)\}_{x\in \Z^{d}}$ is an i.i.d. sequence of random  variables, and 
	$\eps$ represents the amount of disorder with respect to a simple random walk. The proofs involve a criterion that connects 
	localization with the equality or difference between the quenched and annealed rate functions at the boundary.\end{abstract}

\maketitle

\section{Introduction and background}

Random walk in random environment (RWRE) is a fundamental model in probability used as a
prototype for various phenomena. Examples of this include DNA chain replication 
\cite{chernov1967replication}, crystal growth \cite{temkin1969theory}, among others. This model was 
introduced in the '70s to study motion in random media. In dimension $d=1$, the model is well 
understood. Some of the known results include transience, recurrence, law of large numbers 
(\cite{solomon1975rwre},\cite{alili1999rwre}), and large deviations (\cite{greven1994largedeviation}, 
\cite{comets2000quenched}), among others. However, when $d\geq 2$, there are several open questions, 
including how to characterize precisely when the walk is transient or recurrent, or whether directional 
transience implies ballisticity. We refer the reader to the references \cite{drewitz2014selected} and 
\cite{zeitouni2004rwrelectures} for a complete presentation of the model. \\
In this paper, we deal with the notion of localization.
Informally, we say that the walk is \emph{localized} if its asymptotic trajectory is confined to some region 
with positive probability. Otherwise, we say that it is \emph{delocalized}. For RWRE, this has been studied almost 
entirely in the one-dimensional case (see, for example, the works of Sinai \cite{sinai1982rwre} and 
Golosov \cite{golosov1984localization}). When the dimension is two or higher, the topic has been 
practically untouched (\cite{berger2016locallimit} and 
\cite{deuschel2017quenched} are two somewhat related articles). To 
motivate this concept, consider first a simple random walk (SRW)  $(S_{n})_{n\in \N}$ on $\Z^{d}$, 
conditioned to reach the boundary at time $n$, that is, $|S_{n}|_{1}=n$ for each $n\in\N$. This 
walk is an example of delocalization since it presents a diffusive behavior. The natural question is to 
ask if the same situation continues to happen if we perturb the walk in some (random) directions. 
It turns out that the introduction of a small disorder can change the walk's typical paths so that the perturbed 
walk has a \emph{favorite trajectory} that it is likely to visit. That is a 
good reason to study the localization/delocalization phenomena for the RWRE model since the disorder 
can be introduced naturally. In the previous example, we can consider environments of the type 

\begin{equation}\label{eq:srw_pert}
\omega_{\eps}(x,e)=\alpha(e)(1+\eps\xi(x,e))
\end{equation}

\noindent where $(\xi(x,\cdot))_{x\in\Z^{d}}$ is an i.i.d. family of mean-zero random variables, and $
(\alpha(e))_{e\in V}$ are nonnegative numbers such that $\sum_{e\in V}\alpha(e)=1$. Under this 
setting, the question is whether there is localization or delocalization for a given $\eps$. As the case $\eps=0$ 
corresponds to delocalization, one foresees that this will also be the case under a low disorder, and for large 
enough disorder, the opposite may occur. Thus, one might expect the existence of a 
\emph{phase transition} in terms of the parameter $\eps$. That result is proved in 
Proposition \ref{th:main2}. However, the phase transition may be "trivial" in two ways:

\begin{enumerate}
	\item There is only delocalization at $\eps=0$. In other words, the walk is only localized unless it is deterministic. We show in Theorem \ref{th:main1} that this is case if $d=2$ or $3$. Not only that, but any\footnote{More precisely, any RWRE that satisfies both Assumptions \ref{assumpt:iid_ue} and \ref{ass:env_condition}, see below.} (non-deterministic) RWRE will be localized.
	\item There is always delocalization. If $d\geq 4$, we show that the previous situation cannot hold, namely, only localization. Actually, the opposite may take place. Nonetheless, we show in Subsection \ref{sec:example} examples when a genuine phase transition occurs. 
\end{enumerate}

The notion of localization/delocalization is closely related to the equality or difference between the quenched and averaged 
large deviations for RWRE at the boundary. Without being completely rigorous for now, consider a face $F$ of the set 
$\mathbb{D}:=\{x\in \R^{d}:|x|_{1}=1 \}$. If $I_{q}$ and $I_{a}$ are the quenched and annealed rate 
functions for an RWRE (cf. Eq. \ref{eq:quenched_ldp} for the definition), then in 
Theorem \ref{th:loc_equiv_crit} 
we show that localization  in the face  $F$ is equivalent to 

\begin{equation}
\label{eq:inf_rate_funct_ineq} \inf_{x\in F}I_{a}(x)<\inf_{x\in F}I_{q}(x)
\end{equation}
\noindent  and delocalization in the same face corresponds to the equality in \eqref{eq:inf_rate_funct_ineq}. 
 This criterion is one of the crucial results since the annealed 
rate function at the boundary can be computed explicitly (cf. Remark 2.7 in 
\cite{bazaes2019rwredisorder}). Even though the quenched rate function has not an easy explicit formula 
(cf. Theorem 2 in \cite{rosenbluth2006thesis} ), one can obtain estimates for the quenched 
infimum in Eq. \ref{eq:inf_rate_funct_ineq} that ensures the strict inequality in the same 
equation. In Subsection \ref{sec:example} we exploit this fact to show a part of Proposition \ref{th:main2}.

To finish this introduction, we mention that in the model of \emph{directed polymers in random 
environment}, the path localization of the walk has been studied vigorously, and several remarkable 
results have been obtained in the last two decades (cf. \cite{comets2003pathlocalization}, 
\cite{auffinger2011polymerstail}, \cite{bates2018localization} to select a 
few of them). The lectures notes \cite{comets2017directedlectures} contains an updated account of some 
of these articles. 

\subsection{Definitions}
Fix $d\in \N$, the dimension where the walk moves. For $x\in 
\R^{d}$ and $p\in [1,\infty]$, denote by $|x|_p$ its $\ell_p$  norm. Define $V:=\{x\in \Z^{d}:|x|
_{1}=1 \}=\{\pm e_{1},\cdots,\pm e_{d} \}$ the set of allowed jumps of the walk (as usual, $e_{i}$ is the 
vector with zero coordinates excepting the one in the ith position). Next, define $\mathcal{P}$ as the 
set of nearest neighbors probability vectors, that is, 

\begin{equation*}
\mathcal{P}:=\{p:V\to[0,1]:\sum_{e\in V}p(e)=1 \}.
\end{equation*}

\noindent  Now we can define the \emph{environments}. An environment is an element $\omega$ in the space 

\begin{equation*}
\Omega:=\{\omega:\Z^{d}\times V\to [0,1]:\omega(x)\in \mathcal{P}\text{ for all } x\in \Z^{d} \}
=\mathcal{P}^{\Z^{d}}.
\end{equation*} 

\noindent  We usually write $\omega=\{\omega(x,e)\}_{x\in\Z^{d},e\in V}$. Finally, we can define a \emph{random 
walk in the random environment} $\omega\in \Omega$  starting at a point $x\in \Z^{d}$ as the Markov 
chain  $X=(X_{n})_{n\in \N}$ with law $P_{x,\omega}$ that satisfies 

\begin{equation}\label{eq:rwre_def}
\begin{aligned}
P_{x,\omega}(X_{0}=x)&=1,\\
P_{x,\omega}(X_{n+1}=y+e|X_{n}=y)&=\omega(y,e),~n\geq 0, y\in \Z^{d},e\in V.
\end{aligned}
\end{equation}

\noindent  The measure $P_{x,\omega}$ in the literature is known as the \emph{quenched measure}, in contrast to 
the \emph{annealed (or averaged) measure} that we describe next.\\
Equip the space $\Omega$ with the Borel $\sigma-$algebra $\mathbb{B}(\Omega)$, and consider a probability 
measure $\mathbb{P}$ on $(\Omega,\mathbb{B}(\Omega))$. The annealed measure $P_{x}$ of the RWRE starting at $x\in 
\Z^{d}$ is defined as the measure on $\Omega\times (\Z^{d})^{\N}$ that satisfies 

\begin{equation}\label{eq:rwre_ann_measure}
P_{x}(A\times B)=\int_{A}P_{x,\omega}(B)d\mathbb{P}
\end{equation}

\noindent  for each $A\in \mathbb{B}(\Omega)$ and $B\in \mathbb{B}((\Z^{d})^{\N})$, where $\mathbb{B}(\Omega), \mathbb{B}((\Z^{d})^{\N})$ 
are the Borel $\sigma$-algebras of $\Omega$ and $
(\Z^{d})^{\N}$ respectively. Expectations with respect to $P_{x,\omega},P_{x}$ and $\mathbb{P}$   are 
denoted by $E_{x,\omega},E_{x}$ and $\E$ respectively. 
The basics assumptions in this work are the following:

\begin{assumption}\label{assumpt:iid_ue}\leavevmode
 \begin{enumerate}[label=(\roman*)] 
	\item  The random vectors $\{\omega(x,\cdot) \}_{x\in \Z^{d}}$ are i.i.d under $\mathbb{P}$.
	\item  Uniform ellipticity: there exists a $\kappa>0$ such that for every $x\in \Z^{d}$ and $e\in V$,
		\begin{equation}\label{eq:rwre_ue}
			\mathbb{P}(\omega(x,e)\geq \kappa)=1.
		\end{equation}
\end{enumerate}
\end{assumption}
The two assumptions above are common in the literature. In particular, under assumption $(i)$, we can define \begin{equation*}
q(e):=\E[\omega(0,e)]=\E[\omega(x,e)],~~x\in \Z^{d},e\in V.
\end{equation*}

\subsection{Localization at the boundary}
 We will look at trajectories $(X_n)_{n\in \N}$ 
of an RWRE such that $|{X_{n}}|_{1}=n$ for each $n$, and study the asymptotic behavior of the \emph{normalized} quenched probability 
of reaching the boundary at time $n$, that is, if $x\in \Z^{d}$ satisfies $|x|_{1}=n$, 
\begin{equation}\label{eq:quenched_normalized}
P_{0,\omega}\left(X_{n}=x~ \vert ~|X|_{1}=n\right).
\end{equation}

\noindent Specifically, we are concerned in 
knowing if for some sequence $(x_{n})_{n\in \N}\subset \Z^{d}$ such that $|x|_{1}=n$ for all $n$, the 
quenched probability \eqref{eq:quenched_normalized} is greater than some constant $c$, uniformly on 
$n$. In this case, the conditioned walk is "localized" around this path (the rigorous definition 
appears in Definition \ref{def:localization} below). There is a counterpart in the literature of directed 
polymers in random environment (c.f. \cite{comets2017directedlectures}, page 88). In this 
model, there is a nice characterization of localization/delocalization depending on the disorder of the 
environment. For RWRE, the disorder measures how far is the environment $\omega(0,e)$ from its 
expectation $q(e)$. This allows us to obtain analogous results in our case. \\
At this point, we proceed to define localization rigorously. We decompose $\partial\D$ in faces 
$\partial \D(s),s\in \{-1,1\}^{d}$, defined by 

\begin{equation}\label{eq:faces_def}
\partial \D(s):=\{x\in \partial D:s_{j}x_{j}\geq 0,j=1,\cdots, d \}.
\end{equation}

\noindent Without loss of generality, from now on we consider only $\partial \D^{+}:=\partial\D(\overline{s})$, 
where $\overline{s}:=(1,1,\cdots 1)$. We define the allowed jumps by 

\begin{equation*}
V^{+}:=\{e_1,\cdots, e_d\} \subseteq V.
\end{equation*}

\noindent Next, we consider the set 

\begin{equation*}
\partial R_{n}:=n\partial \D^{+}=\{ x\in \Z^{d}: |x|_{1}=n, x_{j}\geq 0\text{ for all } j\in 
\{1,\cdots, d \}\}
\end{equation*}

\noindent and define $R_{n}$ as the sets of all paths $(z_{0},z_{1},\cdots,z_{n})\in (\Z^{d})^{n+1}$ for which 
$z_{0}=0$ and $z_{n}\in \partial R_{n}$. Note that this happens if and only if $\triangle z_{i}:=z_{i}-z_{i-1}\in V^{+}$ 
for each $i=1,\cdots,n$.\\  
\noindent We also consider the sequence  $(J_{n})_{n\in \N}$ defined by $J_{1}:=1$, and for $n\geq 2$,
\begin{equation}\label{eq:loc_Jn_def}
J_{n}:=\max_{x\in \Z^d}P_{0,\omega}(X_{n-1}=x\vert \mathcal{A}_{n-1}),
\end{equation}

\noindent where $\mathcal{A}_{n}:=\{X_{n}-X_0\in \partial R_{n} \}$.\\

\begin{definition}\label{def:localization} Given an RWRE $(X_{n})_{n\in \N}$, we say that it is localized 
at the boundary if \begin{equation}\label{eq:loc_def}
\liminf_{n\to \infty}\frac{1}{n}\sum_{k=1}^{n}J_{n}>0~~\mathbb{P}-a.s.
\end{equation}

\noindent Similarly, the RWRE is delocalized at the boundary if  

\begin{equation}\label{eq:deloc_def}
\liminf_{n\to \infty}\frac{1}{n}\sum_{k=1}^{n}J_{n}=0~~\mathbb{P}-a.s.
\end{equation}
\end{definition}

\noindent Note that \emph{a priori}, the walk can be neither localized nor delocalized. However, in 
Theorem \ref{th:loc_equiv_crit}, we show that this cannot happen for walks that satisfy Assumption 
\ref{assumpt:iid_ue}.

\subsubsection{ A different formulation}
Working on the boundary induces a \emph{polymer-like} interpretation that makes more transparent the argument we use 
below. Given $\omega\in \Omega, x\in \Z^d$, and $e\in V^+$, define 

\begin{equation}\label{def:pi_def}
\pi(\omega,x,e):=\frac{\omega(x,e)}{\sum_{e'\in V^+}
\omega(x,e')},~~\Psi(\omega,x):=\log\left(\sum_{e\in V^+}\omega(x,e)\right)
\end{equation} 

\noindent Then, $\omega(x,e)=\pi(\omega,x,e)e^{\Psi(\omega,x)}$, and $\pi$ induces an RWRE, with $V^+$ as the set of 
allowed jumps. Denote by $P_{x,\pi}$ for the quenched measure starting at $x\in \Z^d$, and $E_{x,\pi}$ its 
expectation. Therefore, for fixed $n\in \N$ and $A\in  \mathbb{B}((\Z^{d})^{\N})$,

\begin{equation}\label{eq:omega_to_pi_rel}
P_{0,\omega}(A, X_n\in \partial R_n)=E_{0,\pi}\left(e^{\sum_{i=0}^{n-1}\Psi(\omega,X_i)}, A\right).	
\end{equation}

\noindent This identity leads to define a quenched \emph{polymer measure} $P_{x,n}^{\omega}$ defined by 

\begin{equation}\label{def:queched_polymer_measure}
P_{x,n}^{\omega}(A):=\frac{E_{0,\pi}\left(e^{\sum_{i=0}^{n-1}\Psi(\omega,X_i)}, A\right)}{E_{0,\pi}\left(e^{\sum_{i=0}
^{n-1}\Psi(\omega,X_i)}\right)},~~ A\in  \mathbb{B}((\Z^{d})^{\N})
\end{equation}

\noindent This resembles the general framework introduced in \cite{rassoul2013quenched}.

\noindent Using the polymer measure, it is direct to verify the identity 
\begin{equation*}
J_n= \max_{x\in \Z^d}P_{0,n-1}^{\omega}(X_{n-1}=x)	
\end{equation*}

\noindent From now on, we use this scheme (except in Subsection \ref{sec:example}), although, of 
course, both definitions are equivalent. 

\subsection{Main results}
The main results of this paper are that localization holds for (almost) all uniformly elliptic and i.i.d 
environments in dimensions two and three, and a phase transition in terms of the disorder in dimensions $d\geq 4$.\\
The following condition will play a remarkable role in the results that follow.

\begin{assumption}\label{ass:env_condition}
The measure $\mathbb{P}$ satisfies 
\begin{equation}\label{eq:env_condition}
\mathbb{P}(\Psi(\omega,0)=\log(c)) <1,
\end{equation}

\noindent where $c:=\sum_{e\in V^+}q(e)$
\end{assumption}

\begin{theorem}\label{th:main1}
 Let $(X_{n})_{n\in \N}$ be an RWRE that satisfies Assumption \ref{assumpt:iid_ue}, and $d\in \{2,3\}$. If 
 Assumption \ref{ass:env_condition} holds, then there is localization. Otherwise, there is delocalization. 
 Delocalization also holds for any dimension $d\geq 2$ when Assumption \ref{ass:env_condition} is not satisfied.
\end{theorem}

A related result in RWRE appears in the article \cite{yilmaz2010differing} of Yilmaz and Zeitouni. 
They show that for walks that satisfy certain ballisticity condition\footnote{In the article above,
it is used the so-called condition (T). This condition is equivalent to the ballisticity conditions (T') 
and $\mathscr{P}_{M}$, as showed in \cite{guerra2020proof}} , there is a class of measures $
\mathbb{P}$, such that the quenched and annealed rate functions differ in a neighborhood of the LLN 
velocity. \\
In the directed polymer model, Comets and Vargas \cite{comets2006majorizing} prove localization in 
dimension $1+1$ (one dimension for time, and one for space), while Lacoin \cite{lacoin2010newbounds} 
proves localization in dimension $1+2$. Berger and Lacoin improved this result in 
\cite{berger2017highTemp}, where they gave the precise asymptotic behavior for the difference between the quenched 
and annealed free energies, as $n\to \infty$.\\

For $d\geq 4$, we consider a certain family of environments, parameterized by $\eps\in [0,1)$. This 
parameter represents how much the distribution of the jumps in an RWRE differs from a simple 
random walk. \\
First, fix a probability vector $\alpha= (\alpha(e))_{e\in V}$ with strictly positive entries. Define 

\begin{equation}\label{def:E_alpha}
\mathcal{E}_{\alpha}:=\big\{(r(e))_{e\in V}\in [-1,1]^{V}:\sum_{e\in V}r(e)\alpha(e)=0,\sup_{e\in V}|r(e)|=1\big\}	
\end{equation}
and consider a probability measure $\Q$ on $\Gamma_{\alpha}:=\mathcal{E}_{\alpha}^{\mathbb{Z}^d}$ (also fixed from 
now). Next, pick an i.i.d family of random variables $(\xi(x))_{x\in \Z^d}\in \Gamma_{\alpha}$ such that $
\E[\xi(x,e)]=0$ for all $e\in V$. Finally, given $\eps\in [0,1)$, define the environments $
(\omega_{\eps}(x))_{x\in \Z^d}$ as 

\begin{equation}\label{def:param_env}
\omega_{\eps}(x,e):=\alpha(e)(1+\eps\xi(x,e))
\end{equation}

\noindent This framework was originally used in \cite{bazaes2019rwredisorder} to study a phase 
transition of the map 

\begin{equation*}
\eps\to I_{a}(x,\cdot)-I_{q}(x,\cdot),
\end{equation*}

\noindent where $I_{q}(x,\cdot), I_{a}(x,\cdot)$ are the quenched (respectively annealed) rate 
functions of an RWRE in the environment $\omega_{\eps}$. 
The study of RWRE at low disorder has also been considered in 
\cite{sznitman2003newexamples}, \cite{sabot2004ballistic}, among others.\\
Notice that for fixed $\eps\in [0,1)$, if we denote by $\mathbb{P}_{\eps}$ to the law of $\omega_{\eps}
$, then this 
measure is uniformly elliptic with constant $\kappa=(1-\eps)\min_{e\in V}\alpha(e)$. Conversely, for 
fixed $\kappa<\frac{1}{\min_{e\in V}\alpha(e)}$, we define  $\eps_{max}:=1-\frac{\kappa}{\min_{e\in V}
\alpha(e)}$, the maximum parameter so that for all $\eps\leq \eps_{max}$, $\mathbb{P}_{\eps}$ is 
uniformly elliptic with constant $\kappa$.\\
The last result of the paper is the phase transition for localization/delocalization for parametrized 
environments. We say that an RWRE is $\eps$-localized (resp. delocalized) if \eqref{eq:loc_def} (resp. 
\eqref{eq:deloc_def}) holds under the measure $\mathbb{P}_\eps$.

\begin{proposition}\label{th:main2}
For $d\geq 2$, $\alpha= (\alpha(e))_{e\in V}$, $\Q$  and $\kappa$ fixed, there exists $\overline{\eps}
\in [0,\eps_{max}]$ such that the walk is $\eps$-localized for $0\leq \eps\leq \overline{\eps}$, and $
\eps$-delocalized for $\overline{\eps}<\eps\leq \eps_{max}$. Moreover,
\begin{enumerate}[label=(\roman*)]
		\item If Assumption \ref{ass:env_condition} does not hold, then $\overline{\eps}=\eps_{max}$. Otherwise,
		\item  if $d=2$ or 3, then $\overline{\eps}=0$;
		\item if $d\geq 4, \overline{\eps}>0$. Also, there are examples of walks that satisfy $\overline{\eps}<\eps_{max}$.
	\end{enumerate}
\end{proposition}

Clearly, $(i)$ and  $(ii)$ are consequence of Theorem \ref{th:main1}. The remaining part combines 
results from \cite{bazaes2019rwredisorder} together with Theorem \ref{th:main1} and the 
example from Subsection \ref{sec:example}. 

\section{An equivalent criterion for localization}\label{sec:loc_equiv_crit}
In this section, we prove an equivalent criterion of localization/delocalization that will be used throughout the sequel. 
First, we need to define the following quantities.

\begin{definition}\label{def:free_energy}
Let $(X_{n})_{n\in \N}$ be an RWRE. Define the limits 

\begin{equation}\label{eq:free_energy}
\begin{aligned}
p(\omega):=&\lim_{n\to \infty}\frac{1}{n}\log E_{0,\pi}\left(e^{\sum_{i=0}
^{n-1}\Psi(\omega,X_{i})}\right),\\ \lambda:=&\lim_{n\to \infty}\frac{1}{n}\log \E E_{0,\pi}
\left(e^{\sum_{i=0}^{n-1}\Psi(\omega,X_{i})}\right)=\log(c),
\end{aligned}
\end{equation}

\noindent where the last equality holds since the conditioned walk is directed.
\end{definition}

\noindent In the directed polymer literature, these limits are known as \emph{quenched and 
annealed} free energy, respectively.\\ We leave the proof of the 
existence of $p(\omega)$ to the end of the section (cf. Lemma \ref{lemma:existence_free_energy}). 
Moreover, we will show that it does not depend on the environment, i.e., it is constant $\mathbb{P}$-
a.s. Hence, assuming the existence and non-randomness of $p$ for now, by Jensen's inequality, we deduce that $p\leq 
\lambda$.

\begin{theorem}\label{th:loc_equiv_crit}
Let $(X_{n})_{n\in \N}$ be an RWRE that satisfies Assumption \ref{assumpt:iid_ue}.
\begin{enumerate}[label=(\roman*)]
	\item The RWRE is localized at the boundary if and only if $p<\lambda$.
	\item The RWRE is delocalized at the boundary if and only if  $p=\lambda$.
\end{enumerate}
In particular,  the walk is either localized or delocalized $\mathbb{P}$- a.s.
\end{theorem}

\subsection{Proof of Theorem \ref{th:loc_equiv_crit}}
In order to prove the result, we need to introduce a couple of definitions. The first is a martingale
that is related to $p$ and $\lambda$, and the second is a quantity linked to $J_{n}$.
\begin{definition}\label{def:Wn_martingale}
Given an RWRE $(X_{n})_{n\in \N}$ that satisfies Assumption \ref{assumpt:iid_ue}, define the random variable in $
(\Omega,\mathbb{B}(\Omega),\mathbb{P})$

\begin{equation}\label{eq:Wn_martingale}
W_{n}(\omega):=E_{0,\pi}\left(e^{\sum_{i=0}^{n-1}\Psi(\omega,X_{i})-n\log(c)}\right),~~n\in 
\N.
\end{equation}
\end{definition}

The following lemma is straightforward, so its proof is skipped. 
\begin{lemma}\label{prop:Wn_martingale}
The process $\{W_{n} \}_{n\in \N}$ is a mean-one $\mathcal{F}_{n}$-martingale under the filtration 
$\{\mathcal{F}_{n}\}_{n\geq 0}$ 	given by $\mathcal{F}_{0}:=\{\emptyset,\Omega \}$, and for $n\geq 
1,\mathcal{F}_{n}:=\{\omega(x,e):|x|_{1}<n, x\in \Z^{d},e\in V^+ \}$. 
\end{lemma}

\noindent The martingale convergence theorem implies that $W_{\infty}:=\lim_{n\to \infty} W_{n}$ exists and is 
non-negative $\mathbb{P}$-a.s. Since the event $\{W_{\infty}=0 \}$ is $T_{e}$-invariant $\mathbb{P}$-
a.s. for each $e\in V^+$, the ergodicity of $\mathbb{P}$ implies that $\mathbb{P}(W_{\infty}
=0)\in \{0,1 \}$. This consequence will be useful in Proposition \ref{prop:Wn_In_relation}.\\
 Next, we introduce a second random variable,

\begin{equation}\label{eq:In_def}
I_{n}(\omega):=\sum_{x\in \Z^d}P_{0,n-1}^{\omega}(X_{n-1}=x)^{2}.
\end{equation}

\noindent This random variable is $\mathcal{F}_{n-1}$-measurable. Observe that 

\begin{equation}\label{eq:In_Jn_relation}
J_{n}^{2}\leq I_{n}\leq J_{n}.
\end{equation}

\noindent These inequalities imply that both $\frac{1}{n}\sum_{k=1}^{n}J_{k}$ and $\frac{1}{n}\sum_{k=1}^{n}I_{k}
$ have the same asymptotics as $n$ goes to infinity.\\
The main ingredient in the proof of Theorem \ref{th:loc_equiv_crit} is the next proposition, which compares 
$W_{n}$ and $I_{n}$.\\We use the following notation: for sequences $(a_{n}),(b_{n})$ we say that 
$a_{n}=\Theta(b_{n})$ if  $a_{n}=O(b_{n})$ and $b_{n}=O(a_{n})$.

\begin{proposition}\label{prop:Wn_In_relation}
Given an RWRE $(X_{n})_{n\in \N}$ that satisfies both Assumption \ref{assumpt:iid_ue} and 
Assumption \ref{ass:env_condition}, the equality

\begin{equation}\label{eq:Wn_In_relation}
\{W_{\infty}=0\}=\left\{ \sum_{n=1}^{\infty}I_{n}=\infty \right\}
\end{equation}

\noindent holds $\mathbb{P}$-a.s. Furthermore, if $\mathbb{P}(W_{\infty}=0)=1$, there exist constants $c_{1}
(\mathbb{P}),c_{2}(\mathbb{P})\in (0,\infty)$ for which $\mathbb{P}$-a.s.,

\begin{equation}\label{eq:Wn_In_relation2}
c_{1}\sum_{k=1}^{n}I_{k}\leq -\log W_{n}\leq c_{2}\sum_{k=1}^{n}I_{k}~~\text{ for } n \text{ large 
enough.}
\end{equation}	 

\noindent That is, $-\log W_{n}=\Theta(\sum_{k=1}^{n}I_{k})$. \end{proposition}

\begin{proof}[Sketch of the proof of Proposition \ref{prop:Wn_In_relation}]
The proof of Theorem 2.1 in \cite{comets2003pathlocalization} can be adapted to show  
Proposition \ref{prop:Wn_In_relation}. It is based on the Doob's decomposition of the submartingale $-\log W_{n}$. 
More precisely, there exist a martingale $\{M_n\}_{n\in \N}$ and an adapted process $\{A_n\}_{n\in \N}$ such that for 
all $n\in \N$,

\begin{equation}\label{def:doob_decomp}
-\log W_n=M_n+A_n.
\end{equation}
Indeed, $A_n:=-\sum_{i=1}^{n}\E\left[\log\left(\frac{W_i}{W_{i-1}}\right)\big\vert \mathcal{F}_{i-1}\right]$. Noting that
\begin{equation*}
\frac{W_i}{W_{i-1}}=E_{0,i-1}^{\omega}\left[e^{\Psi(\omega,X_{i-1})-\log(c)}\right]= 1+ E_{0,i-1}^{\omega}
\left[e^{\Psi(\omega,X_{i-1})-\log(c)}-1\right]=:1+ U_i,  
\end{equation*}

\noindent we decompose $A_n$ and $M_n$ as 

\begin{equation*}
A_n=-\sum_{i=1}^{n}\E[\log(1+U_i)|\mathcal{F}_{i-1}],~~ M_n=\sum_{i=1}^{n} \left(-\log(1+U_i)+\E[\log(1+U_i)|\mathcal{F}
_{i-1}]\right)
\end{equation*}

Exactly as in the aforementioned result, it is enough to prove that there is a constant 
$C>0$ such that for all $n\in \N$,

\begin{equation}\label{eq:e6}
\frac{1}{C}I_n\leq\E[-\log(1+U_n)|\mathcal{F}_{n-1}]\leq C I_n,~~ \E[\log^2(1+U_n)|\mathcal{F}
_{n-1}]\leq C I_n
\end{equation}

\noindent To check the inequalities above, notice that, by uniform ellipticity, the potential $\Psi$ is bounded $
\mathbb{P}$-a.s., so there are constants $0<C_1<C_2$ such that $\mathbb{P}$- a.s., for all $n\in 
\N, \frac{W_n}{W_{n-1}}\in (C_1,C_2)$, and therefore, 
\begin{equation}\label{eq:eq7}
	U_n- C_3 U_{n}^2\leq \log(1+U_n)\leq U_n- C_4 U_{n}^2
\end{equation}

\noindent for some constants $C_3,C_4>0$. Thus, $\E[-\log(1+U_n)|\mathcal{F}_{n-1}]$ is bounded 
by above by

\begin{align*}
&\E[-U_n+C_3 U_n^2|\mathcal{F}_{n-1}]=-C_4\E[ U_n^2|\mathcal{F}_{n-1}]\\
\begin{split}
	& =C_3\sum_{x,x'\in \Z^d}\E\big[E_{0,n-1}^{\omega}\left(e^{\Psi(\omega,x)-\log(c)}-1,X_{n-1}
=x\right)\times\\ &\quad \hspace{2cm} E_{0,n-1}^{\omega}\left(e^{\Psi(\omega,x')-\log(c)}-1,X_{n-1}=x'\right)\big\vert\mathcal{F}
_{n-1}\big]
\end{split}\\
\begin{split}
&=C_3\sum_{x,x'\in \Z^d}\E\left[\left(e^{\Psi(\omega,x)-\log(c)}-1\right)
\left(e^{\Psi(\omega,x')-\log(c)}-1\right)\right]\times \\
&\quad \hspace{2cm} P_{0,n-1}^{\omega}(X_{n-1}=x)P_{0,n-1}^{\omega}(X_{n-1}=x')
\end{split}\\
&=C_3\E\left[\left(e^{\Psi(\omega,0)-\log(c)}-1\right)^2\right]I_n
\end{align*}

\noindent Similarly we get a lower bound $\E[-\log(1+U_n)|\mathcal{F}_{n-1}]
\geq C_4\E\left[\left(e^{\Psi(\omega,0)-\log(c)}-1\right)^2\right]I_n$, and 
this shows the first inequality in \eqref{eq:eq7}. Finally, noting that for 
some constant $C_5 >0$, $\log^2(1+U_n)\leq C_5 U_n^2$, repeating the steps 
from the last display we get the second inequality on \eqref{eq:eq7}, 
concluding the proof.   
\end{proof}

\begin{proof}[Proof of Theorem \ref{th:loc_equiv_crit}]\ \\ 
First recall that  due to \eqref{eq:In_Jn_relation}, we have 

\begin{equation*}
\left(\frac{1}{n}\sum_{k=1}^{n}J_{k}\right)^{2}\leq \frac{1}{n}\sum_{k=1}^{n}J_{k}^{2}\leq 
\frac{1}{n}\sum_{k=1}^{n}I_{k}\leq \frac{1}{n}\sum_{k=1}^{n}J_{k}.
\end{equation*}

\noindent Thus, the liminfs of the sequences $(\frac{1}{n}\sum_{k=1}^{n}I_{k})_{n}$ and $(\frac{1}{n}\sum_{k=1}
^{n}J_{k})_{n}$ are of the same nature, that is, both are positive $\mathbb{P}$-a.s. or zero $
\mathbb{P}$-a.s.\\
If $p<\lambda$, $W_{\infty}=0~\mathbb{P}$-a.s. To check this, observe that if $W_{\infty}>0$ 
then $\frac{\log W_{n}}{n}\to 0$, but at the same time 

\begin{equation*}
\frac{\log W_{n}}{n}\to p-\lambda=0.
\end{equation*}	

\noindent So, if $p<\lambda$, then $W_{\infty}=0~\mathbb{P}$-a.s. and  by 
\eqref{eq:Wn_In_relation}, $\sum_{n}I_{n}=\infty$ a.s. and $-\log W_{n} =\Theta(\sum_{k=1}^{n}I_{k})$. 
In particular, $\liminf_{n\to \infty}\frac{1}{n}\sum_{k=1}^{n}I_{n}>0$, so the RWRE is localized at the boundary. 
Reciprocally, if the walk is localized, $\sum_{k=1}^{n}I_{k}=\infty$, so by \eqref{eq:Wn_In_relation} , 
$-\log W_{n}=\Theta( \sum_{k=1}^{n}I_{k})$ and then $-\frac{\log W_{n}}{n}\to p-\lambda>0$. This proves 
$i)$, and the proof of $ii)$ is analogous.
\end{proof}	

\subsection{Relation between $p$ and $\lambda$ with RWRE rate functions }

To justify the existence of the first limit in \eqref{eq:free_energy}, we relate $p$ (resp. $\lambda$) 
to the quenched (resp. annealed) rate function for random walks in random environment. First, we recall some 
standard notation. We say that the position of the walk satisfies a \emph{quenched large deviation principle} if there 
is a lower semicontinuous function $I_{q}:\R^{d}\to [0,\infty]$ such that for each Borel set $G\subset 
\R^{d}$

\begin{equation}\label{eq:quenched_ldp}
-\inf_{x\in G^{\circ}}I_{q}(x)\leq \liminf_{n\to \infty}\frac{1}{n}\log P_{0,\omega}(X_{n}/n\in G)\leq 
\limsup_{n\to \infty}\frac{1}{n}\log P_{0,\omega}(X_{n}/n\in G)\leq -\inf_{x\in \overline{G}}I_{q}(x).
\end{equation}

\noindent Here $G^{\circ},\overline{G}$ are the interior and closure of $G$ respectively.\\
Analogously, we say that the position of the walk satisfies an \emph{annealed large deviation 
principle} if there is a lower semicontinuous function $I_{a}:\R^{d}\to [0,\infty]$ such that for every 
Borel set $G\subset \R^{d},\eqref{eq:quenched_ldp}$ holds with $P_{0}$ instead of $P_{0,\omega}$. 
It is well known that the domain of both functions (that is, when $I_{q},I_{a}<\infty$) is the set $\mathbb{D}:=\{x\in \R^{d}:|x|_{1}\leq 1 \}$. Also, by Jensen's inequality and Fatou's lemma, $I_{a}\leq I_{q}$.\\
Moreover, Varadhan proved in \cite{varadhan2003ldprwre} that both functions exists under i.i.d and 
uniform elliptic environments, and $I_{q}$ is deterministic (i.e., it does not depend on $\omega$).\\
Next, we characterize the rate functions at $\partial \D^+$ (cf. Eq. \ref{eq:faces_def}).

\begin{lemma}\label{lemma:char_rate_function}
Under Assumption \ref{assumpt:iid_ue}, for any  $x\in \partial \D^+$ there is a sequence $
(x_{n})_{n\in \N}$ such that for all $n, x_{n}\in \Z^{d},|x_n|_{1}=n, \frac{x_n}{n}\to x$, and 

\begin{equation}\label{eq:char_rate_function}
I_{q}(x)=-\lim_{n\to \infty}\frac{1}{n}\log P_{0,\omega}(X_{n}=x_{n}),\hspace{0.5cm} I_{a}(x)=-\lim_{n\to \infty}\frac{1}{n}
\log P_{0}(X_{n}=x_{n}).
\end{equation}
\end{lemma}

\noindent In particular, the limit is independent of the chosen sequence. This result is Lemma 6.5 in 
\cite{bazaes2019rwredisorder}.\\
Finally, the existence of $p$ is consequence of the lemma below.

\begin{lemma}\label{lemma:existence_free_energy}
For an RWRE that satisfies Assumption \ref{assumpt:iid_ue}, the following identities hold:
\begin{equation}\label{eq:existence_free_energy}
p=-\inf_{x\in \partial \D^+}I_{q}(x), \hspace{0.5cm}\lambda=-\inf_{x\in \partial \D^+}I_{a}(x).	
\end{equation}

\noindent In particular, $p$ is not random (since $I_{q}$ is deterministic).
\end{lemma}

\noindent The proof of this lemma is standard (c.f. Lemma 16.12 in \cite{rassoul2015ldpbook}). As a corollary, 
we obtain the characterization of localization/delocalization in terms of the difference between 
the infima of the quenched  and annealed rate functions:

\begin{corollary}\label{cor:cor1}
For an RWRE that satisfies Assumption \ref{assumpt:iid_ue}, we have localization at the boundary if and only if 

\begin{equation*}
\inf_{x\in \partial \D^+}I_{a}(x)<	\inf_{x\in \partial \D^+}I_{q}(x)
\end{equation*}
\end{corollary}
\section{Proof of Theorem \ref{th:main1}}
\subsection{Preliminaries for the proof of Theorem \ref{th:main1}}
The method presented here was used by Lacoin \cite{lacoin2010newbounds}, Berger and 
Lacoin \cite{berger2017highTemp} in the directed polymers model, and by Yilmaz and Zeitouni 
\cite{yilmaz2010differing} for random walks in random environment. As the proofs are similar, 
we only mention the main points of them and refer to the papers above for further details. 
More precisely, in \cite{yilmaz2010differing}, the analog of showing that $p<\lambda$ in the 
space-time RWRE setting is showing that for a sufficiently large set  of  points $\theta\in 
\R^d$, 

\begin{equation}\label{eq:yilmaz_zeitouni_analog}
\lim_{n\to\infty}\frac{1}{n}\E \log E_{0,\omega}\left[e^{\langle \theta, X_n\rangle 
-n\log(\phi(\theta))}\right]<0,
\end{equation}

\noindent where $\phi(\theta):=\sum_{e\in V} q(e) e^{\langle \theta, z\rangle}$. Comparing with 
\begin{equation*}
p-\lambda=\lim_{n\to \infty}\frac{1}{n}\E\log [W_{n}]=\lim_{n\to\infty}\frac{1}{n}\E\log 	E_{0,\pi}
\left(e^{\sum_{i=0}^{n-1}\Psi(\omega,X_{i})-n\log(c)}\right),
\end{equation*}

\noindent the main difference between the two models is that the potential $\Psi(\omega,x)$  
is replaced by a tilt that depends on the steps of the walk, namely, $\Psi_{st}
(\theta,e):=e^{\langle \theta,e\rangle}$. This introduces a correlation that, in our case, is not present (cf. the 
paragraph below Eq. \ref{eq:alpha_def}). Thus, it is natural to apply the methods in 
\cite{yilmaz2010differing} to deduce the desired result. We sketch the main ideas and 
differences in the next pages.

First, note that Theorem \ref{th:loc_equiv_crit} implies immediately delocalization when 
\eqref{eq:env_condition} does not hold. Indeed, in this case, $\mathbb{P}$-a.s $\Psi(\omega,x)=\log(c)$ for all $x\in 
\Z^d$, so by  \eqref{eq:free_energy}, $p=\log(c)=\lambda$. Hence, until the end of the proof we assume that  \eqref{eq:env_condition} holds. \\
We define also $\mu:=\hat{E}(\hat{X}_{1})$, where $\{\hat{X}_n\}_{n\in \N}$ is a simple random walk with jumps in $V^+$ and  
law $\hat{P}$ that satisfies 

\begin{equation*}
\hat{P}(\hat{X}_{n+1}=x+e|\hat{X}_n=x)=\frac{q(e)}{\sum_{e'\in V^+}q(e')},~~~x\in \partial R_n, e\in V^+
\end{equation*}

\noindent Consider $N=nm$ with $n$ fixed (but large enough) and $m\to \infty$. Recall that 

\begin{equation*}
W_{N}(\omega)=E_{0,\pi}\left(e^{\sum_{i=0}^{N-1}\Psi(\omega,X_{i})-N\log(c)}\right)
\end{equation*} 

\noindent We define, for $y\in \Z^{d}$,

\begin{equation}\label{def:Jy}
J_{y}:=\left((y-\frac{1}{2})\sqrt{n},(y+\frac{1}{2})\sqrt{n}\right)\subset \R^{d}.
\end{equation}

\noindent Given $Y=(y_{1},\cdots,y_{m})\in (\Z^{d})^{m}$, we decompose 

\begin{equation}\label{eq:Wn_decomp}
W_{N}(\omega)=\sum_{Y}W_{N}(\omega,Y),
\end{equation}

\noindent where 

\begin{equation*}
W_{N}(\omega,Y):=E_{0,\pi}\left(e^{\sum_{i=0}^{N-1}\Psi(\omega,X_{i})-N\log(c)},X_{jn}- jn\mu\in 
J_{y_{j}},\forall j\leq  m\right)
\end{equation*}

\noindent The decomposition in \eqref{eq:Wn_decomp} is valid, since $\Z^{d}\subset \bigcup_{y\in \Lambda}J_{y}$. 
By the inequality $(\sum_{i}a_{i})^{1/2}\leq \sum_{i}a_{i}^{1/2}$, valid for countable indices, we 
obtain 

\begin{equation*}
\E[W_{N}(\omega)^{1/2}]\leq \sum_{Y}\E[W_{N}(\omega,Y)^{1/2}].
\end{equation*}

\noindent This inequality gives us  

\begin{equation}\label{eq:p-lambda_ineq}
p-\lambda=\lim_{N\to \infty}\frac{1}{N}\E\log [W_{N}]\leq \liminf_{N\to \infty}\frac{2}{N}\log 
\E[W_{N}^{1/2}]\leq \liminf_{N\to \infty}\frac{2}{N}\log \left(\sum_{Y}\E[W_{N}(\omega,Y)^{1/2}]
\right).
\end{equation}

\noindent Now we estimate each expectation $\E[W_{N}(\omega,Y)]^{1/2}$, applying the change of measure. The plan 
is the following (recall that $N=mn$ with fixed $n$): for $j\in \{1,\cdots,m\}$, $Y$ fixed and $n$ a square integer, we define 

\begin{equation}
\label{eq:Bj_def}
B_{j}:=\{(z,i)\in \Z^{d}\times \N:(j-1)n\leq i<jn,|z-i\mu -y_{j-1}\sqrt{n}|\leq C_{1}\sqrt{n} \},
\end{equation} 

\noindent where $C_{1}$ is a constant to determine and $y_{0}:=0$.

\subsection{Proof in the case $d=2$}\label{sec:proof_dim_2}

The idea is to define a function that depends on the different blocks $B_{j}$. We define 
\begin{equation}\label{def:D(Bj)}
D(B_{j}):=\sum_{y:(y,|y|_{1})\in B_{j}}\tilde{\omega}(y),~~\text{where}~~~~~\tilde{\omega}
(y):=e^{\Psi(\omega,y)}-c.
\end{equation}

\noindent In particular, $\E[D(B_{j})]=0$, and they form an independent family of random variables. It is important to observe that 
\eqref{ass:env_condition} guarantees that $\tilde{\omega}$ and $D(B_{j})$ are non-degenerate random variables.
We also define $\delta_{n}:=C_{1}^{-1/2}n^{-3/4}$. Note that $\delta_{n}^{2}|D(B_{1})|=O(1)$. Finally, for $K>0$ large 
enough (to determine), define 

\begin{equation*}
f_{K}(u):=-K\mathbbm{1}_{\{u\geq e^{K^{2}} \}},\hspace{0.5cm} g(\omega,Y):=e^{\sum_{j=1}^{m}{f_{K}(\delta_{n}D(B_{j}))}}
\end{equation*}

\noindent By Cauchy-Schwarz inequality, \begin{equation}\label{eq:estim1}
	\begin{aligned}
	&\E[W_{N}(\omega,Y)^{1/2}]=\E[W_{N}(\omega,Y)^{1/2}g(\omega,Y)^{1/2}g(\omega,Y)^{-1/2}]\\&\leq \E[W_{N}
	(\omega,Y)g(\omega,Y)]^{1/2}\E[g(\omega,Y)^{-1}]^{1/2}.
\end{aligned}
\end{equation}

\noindent One can show that for $K$ large enough, $\E[g(\omega,Y)^{-1}]^{1/2}\leq 2^{m}$. To bound $\E[W_{N}
(\omega,Y)g(\omega,Y)]$, we can follow word by word the estimates in Pages 251-252 from \cite{yilmaz2010differing}, to deduce  that 

\begin{equation*}
\E[W_{N}(\omega,Y)^{1/2}]\leq \left(2\sum_{y\in \Z^{2}}\max_{x\in J_{0}}\E E_{x,\pi}\left(e^{\sum_{i=0}
^{n-1}\Psi(\omega,X_i)+f_{K}(\delta_{n}D(B_{1}))-n\log(c)};X_{n}-n\mu\in J_{y} \right)\right)^{m}.
\end{equation*}

\noindent The bound \eqref{eq:p-lambda_ineq} tell us that $p-\lambda<0$ once we are able to prove the following: 

\begin{lemma}
\label{lemma:lemma1}
For $n,K$, and $C_{1}$ large enough,

\begin{equation*}
\sum_{y\in \Z^{2}}\max_{x\in J_{0}}\E E_{x,\pi}\left(e^{\sum_{i=0}^{n-1}\Psi(\omega,X_i)+f_{K}(\delta_{n}
D(B_{1}))-n\log(c)};X_{n}-n\mu\in J_{y} \right)<1/2.
\end{equation*}
\end{lemma}

\noindent The proof of the lemma above is followed almost exactly from Subsection 2.5 in 
\cite{yilmaz2010differing}. The main difference rests in display (2.22) in the aforementioned paper. In our case, 
we need to check that for some $\alpha>0$, 
\begin{equation*}
\E E_{0,\pi}\left[e^{\sum_{i=0}^{n-1}\Psi(\omega,X_{i})-n\log(c)}\left(\sum_{i=0}^{n-1}\tilde{\omega}(X_{i})-
\alpha\right)^{2}\right]
\end{equation*}
is $O(n)$.

\noindent We can decompose it as 

\begin{equation}\label{eq:eq10}
	\begin{aligned}
	&\sum_{j=1}^{n-1}\E E_{0,\pi}\left[e^{\sum_{i=0}^{n-1}\Psi(\omega,X_{i})-n\log(c)}(\tilde{\omega}(X_{j})-
	\alpha)^2\right]+\\ 
	&2\sum_{1\leq \ell<j\leq n-1}\E E_{0,\pi}\left[e^{\sum_{i=0}^{n-1}\Psi(\omega,X_{i})-n\log(c)}(\tilde{\omega}
	(X_{\ell})-\alpha)(\tilde{\omega}(X_{j})-\alpha)\right]
	\end{aligned}
\end{equation}
	
\noindent The first term is  $n\E E_{0,\pi}\left[e^{\Psi(\omega,X_1)-\log(c)}\left(\tilde{\omega}(X_1)-
\alpha\right)^2\right]$. As $c_{n}:=\frac{\delta_{n}^{2}}{(\mu n\delta_{n}-A_{n}-e^{K^{2}})^{2}}=O(n^{-2})$, this 
expression vanishes as $n\to \infty$.  On the other hand, if we choose 

\begin{equation}\label{eq:alpha_def}
\alpha:=\E E_{0,\pi}\left[e^{\Psi(\omega,X_1)-\log(c)}\tilde{\omega}(X_1)\right]=\frac{\E\left[e^{2\Psi(\omega,0)}
\right]-c^2}{c}>0~ \text{ by } \eqref{ass:env_condition},
\end{equation} 

\noindent then by independence the second term in \eqref{eq:eq10} is zero. By comparison, the analog of $\alpha$ 
(called $\mu$ in \cite{yilmaz2010differing}) is greater than zero due to a positive correlation that in our case is not 
needed.

Combining the previous results, such election of constants help us to deduce that Lemma 
\ref{lemma:lemma1} is true, and therefore $p-\lambda<0$.\qed

\subsection{Proof in case $d=3$}\label{sec:proof_dim_3}

In this case, the proof in principle is essentially the same, but some technical details needs to be adapted to 
this situation. In particular, we need to redefine $\delta_n$ and $D(B_j)$, namely

\begin{equation*}
\delta_n:=n^{-1}(\log n)^{-1/2},\hspace{0.5cm} D(B_{j}):=\sum_{\substack{y,z\\ (y,i),(z,j)\in B{j}}}V(y,z)\tilde{\omega}(y)
\tilde{\omega}(z),
\end{equation*}

\noindent where $\tilde{\omega}$ is defined as in \eqref{def:D(Bj)}, and 
\begin{equation*}
V(y,z):=\frac{1}{| i-j |}\mathbbm{1}_{\{  |y-z- (i-j)\mu|<C_{2}\sqrt{|i-j|}  \}}\text{ if } i\not=j,\text{ and } 
0\text{ otherwise},
\end{equation*}

\noindent for some constant $C_{2}$ to determine. The proof of Theorem 1.6 in \cite{yilmaz2010differing} can be 
followed almost word by word, and our case is a little bit simplified since the correlation issue is 
not present, as in the $d=2$ case. Details are omitted.\qed

\section{Phase transition }\label{sec:proof_dim_4}

Recall the parametrization of the environments $(\omega_{\eps})_{\eps\in [0,1)}$ (cf. Eq. \ref{def:param_env}). We 
will denote by $p(\eps)$ to the limit in \eqref{eq:free_energy} with environment $\omega_{\eps}$. On the other hand, 
$\lambda$ is constant over $\eps$, and it is equal to $\log(\sum_{e\in V^{+}}\alpha(e))$.  
The first part of Proposition \ref{th:main2} is consequence of the lemma below:

\begin{lemma}\label{lemma:d_4:1}
For each $n\in \N$, the map \begin{equation*}
\eps\in [0,\eps_{max}]\to \frac{1}{n}\left[\E\log P_{0,\omega_{\eps}}(X_{n}\in \partial R_{n})-\log P_{0}(X_{n}\in 
\partial R_{n}) \right]\text{ is non-increasing.}
\end{equation*} 
\end{lemma}

\noindent This is an easy adaptation of Lemma 8.1 in \cite{bazaes2019rwredisorder}. If we let $n$ to infinity, then 
we deduce that $p(\eps)-\lambda(\eps)$ is non-increasing. To finish the proof, define 

\begin{equation*}
\overline{\eps}:=\inf\{ \eps\in (0,\eps_{max}]:p(\eps)-\lambda(\eps)<0 \},
\end{equation*}

\noindent with the convention that $\inf\emptyset=\eps_{max}$.\\
The rest of this section is devoted to prove  $(iii)$ of Proposition \ref{th:main2}. The main ingredient to show  
the first part in $(iii)$ is the next lemma, a particular case of Lemma 6.1 with $\theta=0$ in 
\cite{bazaes2019rwredisorder}.

\begin{lemma}\label{lemma:d_4:2}
If $\eps>0$ is small enough, then $\sup_{n}|W_{n}^{2}|_{2}<\infty$.
\end{lemma}

\noindent Recall the following:

\begin{equation*}
W_{\infty}(\eps):=W_{\infty}(\omega_{\eps})>0\to p(\eps)=\lambda(\eps)\leftrightarrow \text{delocalization}.
\end{equation*}

\noindent Indeed, If $W_{\infty}>0$, then $\log(W_{\infty})=\lim_{n\to \infty}\log(W_{n})<\infty$, so 

\begin{equation*}
p(\eps)=\lim_{n\to \infty}\frac{1}{n}\log P_{0,\omega_{\eps}}(X_{n}\in \partial R_{n})=\lim_{n\to \infty}\frac{W_{n}
(\omega_{\eps})}{n}+\lambda(\eps)=\lambda(\eps).
\end{equation*}

\noindent Now pick $\eps>0$ small enough such that $\sup_{n}|W_{n}^{2}|_{2}<\infty$ as in Lemma \ref{lemma:d_4:2}, and 
call it $\eps^{*}$. By the martingale convergence theorem, $W_{n}(\eps^{*})\to W_{\infty}(\eps^{*})$ a.s. and in 
$L^{2}$. As $|W_{n}|_{2}=1$ for all $n$, then we necessarily have $W_{\infty}(\eps^{*})>0$, and therefore 
$p(\eps^{*})=\lambda(\eps^{*})$. But the map $\eps\to p(\eps)-\lambda(\eps)$ is non-increasing, so $p=\lambda$ on $
[0,\eps^{*}]$, and thus $\overline{\eps}\geq\eps^{*}>0$.\\

\subsection{An example on which $\overline{\eps}<\eps_{max}$}\label{sec:example}

For simplicity, we consider $d=4$, and i.i.d random variables $(\xi(x))_{x\in \Z^d}\in \Gamma_{\alpha}$ such that $
\xi(x,e)=\xi(x,e')$ for all $e,e'\in V^+$ and $\xi(x,-e)=-\xi(x,e)$. Also, for $i=1,\cdots, d$, define $
\alpha(e_i)=\alpha(-e_i):=\frac{y_i}{2\sum_{i=1}^{d}y_i}$, where $y=(y_1,\cdots,y_d)\in \partial \D^+$ is a point to determine. 
Therefore,  we have 

\begin{equation*}
\omega_\eps(x,e_i)=\alpha(e_i)(1+\eps\xi(x))~~~~e_i\in V^+
\end{equation*}

\noindent Moreover, assume that the distribution of $\xi(0)$ under $\Q$ is the Rademacher distribution, namely, $
\Q(\xi(0)=1)=\Q(\xi(0)=-1)=\frac{1}{2}$. By Corollary \ref{cor:cor1}, localization occurs as soon as

\begin{equation}\label{eq:eq13}
\inf_{x\in \partial \D^+}I_{a}(x)<	\inf_{x\in \partial \D^+}I_{q}(x)
\end{equation}
However, in this case, the infimum on the left is exactly $I_a(y)$, and it is achieved only at this point. On 
the other hand, by the continuity of $I_q$, the infimum on the right is also achieved at some point $\overline{x}\in 
\partial\D^+$. If $\overline{x}\not=y$, then $I_a(y)<I_a(\overline{x})\leq I_q(\overline{x})$ so we are done. Thus, 
we assume that $\overline{x}=y$. In that case, we decompose $-I_q(y)$ as 

\begin{align}
-I_q(y)&=-I_a(y)+\lim_{n\to\infty}\frac{1}{n}\log \left(\frac{\sum_{0=x_0,x_1,\cdots,x_n=y_n}\prod_{i=1}^{n}
q(\triangle x_i)(1+\eps\xi(x_{i-1}))}{\sum_{0=x_0,x_1,\cdots,x_n=y_n}\prod_{i=1}^{n}q(\triangle x_i)}\right)
\nonumber\\
&\leq -I_a(y)+\limsup_{n\to\infty}\max_{0=x_0,x_1,\cdots,y_n}\frac{1}{n}\sum_{i=1}^{n}\log(1+\eps\xi(x_{i-1}))
\label{eq:eq14},
\end{align}

\noindent where $(y_n)_{n\in \N}$ is any sequence such that $y_n\in \partial R_n$ and $\frac{y_n}{n}\to y$ as $n\to 
\infty$. 
Also, the sum and maximum above are over all directed paths $0=x_0,x_1,\cdots,x_n $ such that $x_n=y_n$. Denote by 
$C(y_n)$ to the number of such paths. It's easy to check that there exists some constant $C>0$ such that, for all 
$n\in \N,C(y_n)\leq Ce^{nf(y)-\frac{d-1}{2}\log n}$, where $f(y)=-\sum_{i=1}^{d}y_i\log(y_i)$. To estimate the maximum above, we can use 
Hoeffding inequality (cf. Theorem 
2.8 in \cite{boucheron2013concentration}) to get, for $a>0$,
 \begin{equation}\label{eq:eq15}
\mathbb{P}\left(\sum_{i=1}^{n}\log(1+\eps\xi(x_{i-1}))-n\E[\log(1+\eps\xi(0))]>na\right)\leq 
\exp\left(\frac{-2na^2}{\log\left(\frac{1+\eps}{1-\eps}\right)^2}\right)
\end{equation} 

\noindent Therefore, 

\begin{align*}
&\sum_{n=1}^{\infty}\mathbb{P}\left(\max_{0=x_0,x_1,\cdots,y_n}\sum_{i=1}^{n}\log(1+\eps\xi(x_{i-1}))-n\E[\log(1+
\eps\xi(0))]>na\right)\\
&\leq \sum_{n=1}^{\infty} C(y_n)\exp\left(\frac{-2na^2}{\log\left(\frac{1+\eps}{1-\eps}\right)^2}\right)<\infty
\end{align*}

\noindent as soon as $a>\log\left(\frac{1+\eps}{1-\eps}\right)\sqrt{f(y)/2}$. By Borel-Cantelli's lemma, 
\eqref{eq:eq14} is bounded by \begin{align*}
&-I_a(y)+\log\left(\frac{1+\eps}{1-\eps}\right)\sqrt{f(y)/2}+\E\left[\log(1+\eps\xi(0))\right]\\
&=-I_a(y)+\log\left(\frac{1+\eps}{1-\eps}\right)\sqrt{f(y)/2}+\frac{1}{2}\left(\log(1+\eps)+\log(1-\eps)\right)
\end{align*}
If $f(y)\leq \frac{9}{50}$, then $\sqrt{f(y)/2}\leq \frac{3}{10}$, and the last display is strictly smaller than $-
I_a(y)$ at least for $\eps>\frac{9}{10}$. The required value for $f(y)$ can be achieved, for example, choosing the 
vector $y=\left(\frac{97}{100},\frac{1}{100},\frac{1}{100},\frac{1}{100}\right)$, so in this case, we can choose $
\eps_{max}\approx\frac{9}{10}$ and have a true phase transition, with  $\kappa\approx\frac{1}{1000}$.\qed

\begin{remark}
The asymmetry in terms of $\alpha$ is needed. Indeed, if $\alpha(e)=\frac{1}{2d}$ for all $e\in V$, then it is 
not difficult to show (cf. pp. 36-37 in \cite{comets2017directedlectures}) that under our setting, $\sup_{n\in \N}
\E[W_n^2]<\infty$, and therefore, $p(\eps)=\lambda$ for all $\eps\in [0,\eps_{max}]$.
\end{remark}

\providecommand{\bysame}{\leavevmode\hbox to3em{\hrulefill}\thinspace}
\providecommand{\MR}{\relax\ifhmode\unskip\space\fi MR }
\providecommand{\MRhref}[2]{%
  \href{http://www.ams.org/mathscinet-getitem?mr=#1}{#2}
}
\providecommand{\href}[2]{#2}



\noindent{\bf Acknowledgement.}The author thanks his advisor Alejandro F. Ram\'{\i}rez and an anonymous referee  for valuable comments and suggestions about the paper.


\end{document}